\documentclass[11pt]{amsart}
\usepackage{xcolor}
\usepackage{amstext}
\usepackage{amsfonts}
\usepackage{amssymb}
\usepackage{amsbsy}
\usepackage{latexsym}
\usepackage[T1]{fontenc}
\usepackage{xy}
\usepackage{hhline}
\xyoption{all}
\newcounter{num}[section] %

\newenvironment{theo}
{\refstepcounter{num}%
\bigskip\noindent{\bf Theorem~\arabic{section}.\arabic{num}. }\it}
{\smallskip}

\newenvironment{prop}
{\refstepcounter{num}%
\bigskip\noindent{\bf Proposition~\arabic{section}.\arabic{num}. }\it}

\newenvironment{cor}
{\refstepcounter{num}%
\bigskip\noindent{\bf Corollary~\arabic{section}.\arabic{num}. }\it}

\newenvironment{lemma}
{\refstepcounter{num}%
\bigskip\noindent{\bf Lemma~\arabic{section}.\arabic{num}. }\it}

\newenvironment{eq}{\begin{equation}}{\end{equation}}

\newcommand{\Ref}[1]{(\ref{#1})}

\newcommand{\si}{\sigma}
\newcommand{\al}{\alpha}
\newcommand{\be}{\beta}
\newcommand{\ga}{\gamma}

\newcommand{\un}[1]{{\underline{#1}} }

\newcommand{\tr}{\mathop{\rm tr}}

\newcommand{\mdeg}{\mathop{\rm mdeg}}

\newcommand{\algA}{\mathcal{A}}    

\newcommand{\N}{\mathcal{N}} 

\newcommand{\FF}{{\mathbb{F}}}   
\newcommand{\NN}{{\mathbb{N}}}

\newcommand{\mylabel}[1]{}

\newcommand{\mycomment}[1]{}
\begin{document}
\title[Separating invariants of three nilpotent $3\times3$ matrices]{Separating invariants of three nilpotent $3\times3$ matrices}

\thanks{The first author was partially supported by CAPES and the second author was supported by CNPq 307712/2014-1 and FAPESP 2018/23690-6}

\author{Felipe Barbosa Cavalcante}
\address{Felipe Barbosa Cavalcante\\ 
State University of Campinas, 651 Sergio Buarque de Holanda, 13083-859 Campinas, SP, Brazil}
\email{felipelamerck@hotmail.com (Felipe Barbosa Cavalcante)}

\author{Artem Lopatin}
\address{Artem Lopatin\\ 
State University of Campinas, 651 Sergio Buarque de Holanda, 13083-859 Campinas, SP, Brazil}
\email{dr.artem.lopatin@gmail.com (Artem Lopatin)}

\begin{abstract} 
The algebra $\mathcal{O}(\N_n^d)^{GL_n}$ of $GL_n$-invariants of $d$-tuple of $n\times n$ nilpotent matrices with respect to the action by simultaneous conjugation is generated by the traces of products of nilpotent generic matrices in the case of an algebraically closed  field of characteristic zero. We describe a minimal separating set for this algebra in case $n=d=3$.   


\end{abstract}

\maketitle

\section{Introduction}\label{section_intro}

\subsection{Algebras of invariants} All vector spaces, algebras, and modules are over an algebraically closed field $\FF$ of characteristic zero unless otherwise stated.  By an algebra we always mean an associative algebra with unity.

Given $n>1$ and $d\geq 1$, the direct sum $M_n^d$ of $d$ copies of the space of $n\times n$ matrices over $\FF$ is a $GL_n$-module with respect to the diagonal action by conjugation: $g\cdot\un{A} = (g A_1 g^{-1},\ldots,g A_d g^{-1})$ for $g\in GL_n$ and $\un{A}=(A_1,\ldots,A_d)$ from $M_n^d$.  Given a matrix $A$, denote by $A_{ij}$ the $(i,j)^{\rm th}$ entry of $A$. Any element of the coordinate ring 
$$
\mathcal{O}(M_n^d)  = \FF[x_{ij}(k) \;| \;1\leq i,j\leq n, 1\leq k\leq d]
$$
of $M_n^d$ can be considered as the polynomial function $x_{ij}(k): M_n^d\to \FF$, which sends $(A_1,\ldots,A_d)$ to  $(A_k)_{ij}$.  Sibirskii~\cite{Sibirskii_1968} and Procesi~\cite{Procesi76} established that the algebra of invariants 
$$\mathcal{O}(M_n^d)^{GL_n}=\{f\in \mathcal{O}(M_n^d)\,|\,f(g\cdot \un{A})=f(\un{A}) \text{ for all }g\in GL_n,\; \un{A}\in M_n^d\}$$ 
is generated by $\tr(X_{i_1}\cdots X_{i_r})$, where $X_k$ stands for the $n\times n$ {\it generic} matrix $(x_{ij}(k))_{1\leq i,j\leq n}$. Note that over an infinite field of positive characteristic a generating set for $\mathcal{O}(M_n^d)^{GL_n}$ was described by Donkin~\cite{Donkin92a}.

Denote by $\N_n$ the affine variety of all $n\times n$ {\it nilpotent} matrices, i.e., $A\in M_n$ belongs to $\N_n$ if and only if $\tr(A)=\tr(A^2)=\cdots=\tr(A^n)=0$, or equivalently, $A^n=0$. The variety $\N_n^d$ is known to be irreducible. The coordinate ring $\mathcal{O}(\N_n^d)$ of $\N_n^d\subset M_n^d$  is generated by $y_{ij}(k)$ for $1\leq i,j\leq n$ and $1\leq k\leq d$, where  $y_{ij}(k):\N_n^d\to \FF$ sends $(A_1,\ldots,A_d)$ to $(A_k)_{ij}$. 
Denote by $J=J_{n,d}$ the ideal of all $f\in \mathcal{O}(M_n^d)$ that are zero over $\N_n^d$, i.e., $J=I(\N_n^d)$. Then $\tr(X_k^s)\in J$ for all $1\leq k,s\leq n$. We have 
$$\mathcal{O}(\N_n^d)=\mathcal{O}(M_n^d)/ J$$
and $y_{ij}(k) = x_{ij}(k) + J$. The algebra $\mathcal{O}(\N_n^d)^{GL_n}$ of {\it invariants of nilpotent matrices} is defined in the same way as $\mathcal{O}(M_n^d)^{GL_n}$. Since 
$$0 \longrightarrow J \longrightarrow  \mathcal{O}(M_n^d) \overset{\Psi}{\longrightarrow}  \mathcal{O}(\N_n^d) \longrightarrow 0$$
is a short exact sequence, then
$$0 \longrightarrow J^{GL_n} \longrightarrow  \mathcal{O}(M_n^d)^{GL_n} \overset{\hat{\Psi}}{\longrightarrow} \mathcal{O}(\N_n^d)^{GL_n} \longrightarrow 0$$
is also a short exact sequence, since $GL_n$ is a linearly reductive group. Therefore, the algebra $\mathcal{O}(\N_n^d)^{GL_n}$ is generated by $\tr(Y_{i_1}\cdots Y_{i_r})$, where $Y_k=(y_{ij}(k))_{1\leq i,j\leq n}$ is the {\it generic nilpotent} $n\times n$ matrix. Obviously, $\tr(Y_k^s)=0$ for any $s>0$ and $Y_k^n=0$. Let us remark that in case $d=1$ we have $\mathcal{O}(\N_n^1)^{GL_n}=\FF$.

\subsection{Separating invariants}
\label{section_separ}
In 2002 Derksen and Kemper~\cite{DerksenKemper_book} (see~\cite{DerksenKemper_bookII} for the second edition) introduced the notion of separating invariants as a weaker concept than generating invariants.  Assume that $W$ is $M_n^d$ or $\N_n^d$. Given a subset $S$ of $\mathcal{O}(W)^{GL_n}$, we say that elements $u,v$ of $W$ {\it are separated by $S$} if  exists an invariant $f\in S$ with $f(u)\neq f(v)$. If  $u,v\in W$ are separated by $\mathcal{O}(W)^{GL_n}$, then we simply say that they {\it are separated}. A subset $S\subset \mathcal{O}(W)^{GL_n}$ of the invariant ring is called {\it separating} if for any $v, w$ from $W$ that are separated we have that they are separated by $S$. We say that a separating set is minimal if it is minimal w.r.t.~inclusion. Obviously, any generating set is also separating. Thus the cases when some separating set does not generate the algebra of invariants are of special interest. 

For a monomial $c\in \mathcal{O}(M_n^d)$ denote by $\deg{c}$ its {\it degree} and by $\mdeg{c}$ its {\it multidegree}, i.e., $\mdeg{c}=(t_1,\ldots,t_d)$, where $t_k$ is the total degree of the monomial $c$ in $x_{ij}(k)$, $1\leq i,j\leq n$, and $\deg{c}=t_1+\cdots+t_d$. Then the algebras $\mathcal{O}(M_n^d)^{GL_n}$ and $\mathcal{O}(\N_n^d)^{GL_n}$ have $\NN$-grading by degrees and $\NN^d$-grading by multidegrees, where $\NN$ stands for the set of non-negative integers. We say that a multidegree $(t_1,\ldots,t_d)$ is less than a multidegree $(k_1,\ldots,k_d)$ if $t_i\leq k_i$ for all $i$ and $t_j< k_j$ for some $j$.

Given an algebra of invariants $\mathcal{O}(W)^{GL_n}$, denote by $D_{\rm gen}$ ($D_{\rm sep}$, respectively) the minimal integer $D$ such that the set of all homogeneous invariants of $\mathcal{O}(W)^{GL_n}$ of degree less or equal to $D$ is a generating set (separating set, respectively). Note that $D_{\rm gen}$ is the maximal degree of elements of any minimal (w.r.t.~inclusion) generating set. Since $\hat{\Psi}$ is surjective, we obtain that 
$$  D_{\rm gen}(\mathcal{O}(\N_n^d)^{GL_n}) \leq D_{\rm gen}(\mathcal{O}(M_n^d)^{GL_n}) \;\;\;\text{and}\;\;\; 
D_{\rm sep}(\mathcal{O}(\N_n^d)^{GL_n}) \leq D_{\rm sep}(\mathcal{O}(M_n^d)^{GL_n}).
$$
 
In case of an infinite field of arbitrary characteristic the following minimal separating set for $\mathcal{O}(M_2^d)^{GL_2}$ was given by Kaygorodov, Lopatin, Popov~\cite{Lopatin_separating2x2}:
\begin{eq}\label{eq0}
\tr(X_i^2),\, 1\leq i\leq d;\; 
\tr(X_{i_1}\cdots X_{i_k}),\, k\in\{1,2,3\},\, 1\leq i_1<\cdots<i_k\leq d.
\end{eq}%
Note that set~\Ref{eq0} generates the algebra $\mathcal{O}(M_2^d)^{GL_2}$ if and only if the characteristic of $\FF$ is different from two or $d\leq 3$ (see~\cite{Procesi_1984,DKZ_2002}). A minimal generating set for $\mathcal{O}(M_3^d)^{GL_3}$ was given by Lopatin in~\cite{Lopatin_Sib, Lopatin_Comm1,Lopatin_Comm2} in the case of an arbitrary infinite field. Over a field of characteristic zero a minimal generating set for $\mathcal{O}(M_n^2)^{GL_n}$ was established by Drensky and Sadikova~\cite{Drensky_Sadikova_4x4} in case $n=4$ (see also~\cite{Teranishi_1986}) and by \DJ{}okovi\'c~\cite{Djokovic2007} in case $n=5$ (see also~\cite{Djokovic2009}). 

For the algebra $\mathcal{O}(M_n^d)^{GL_n}$ Derksen and Makam~\cite{DerksenMakam4}  established that $D_{\rm gen}\leq (d+1)n^4$ and $D_{\rm sep}\leq n^6$ in case of an arbitrary algebraically closed field. The second bound was improved in~\cite{DerksenMakam5}, where it was proven that $D_{\rm sep}\leq 4n^2 \log_2(n) + 12 n^2 - 4 n$. Moreover, over an algebraically closed field of zero characteristic it was shown in~\cite{DerksenMakam5} that  $D_{\rm sep}\leq 4n \log_2(n) + 12 n - 4$. 

We can consider $M_n^d$ as an $SL_n \times SL_n$-module with respect to the diagonal left-right action:
$(g_1,g_2)\cdot \un{A} = (g_1 A_1 g_2^{-1},\ldots,g_1 A_d g_2^{-1})$ for all $g_1,g_2\in SL_n$ and $\un{A}\in M_n^d$. A minimal separating set for the algebra of matrix semi-invariants $\mathcal{O}(M_2^d)^{SL_2\times SL_2}$ was explicitly described by Domokos~\cite{Domokos2020} over an arbitrary algebraically closed field. Note that a minimal generating set for $\mathcal{O}(M_2^d)^{SL_2\times SL_2}$ was given by Lopatin~\cite{Lopatin_semi2222} over an arbitrary infinite field (see also~\cite{Domokos2020Add}). Upper bounds on $D_{\rm gen}$ and $D_{\rm sep}$ were given by Derksen and Makam in~\cite{DerksenMakam2, DerksenMakam5, DerksenMakam6}. Note that there is a surjective homomorphism  $\sigma^{\ast}:\mathcal{O}(M_n^d)^{SL_n\times SL_n} \to \mathcal{O}(M_n^{d-1})^{GL_n}$ of algebras sending separating sets to separating sets (see Corollary 6.3 of~\cite{Domokos2020}).

\subsection{Results}
The main result of our paper is the description of a minimal generating set and a minimal separating set for the algebra of invariants  $\mathcal{O}(\N_3^3)^{GL_3}$  (see Theorem~\ref{theo3x3three}). We also show that $D_{\rm sep}(\mathcal{O}(\N_3^d)^{GL_3})=6$ for all $d\geq2$ (see Corollary~\ref{cor}).

The paper is organized as follows. In Section~\ref{section_two_cases} the cases of $\mathcal{O}(\N_2^d)^{GL_2}$ for $d>0$ and  $\mathcal{O}(\N_3^2)^{GL_3}$ are considered. The proof of  Theorem~\ref{theo3x3three} is divided into several lemmas. The statement of Theorem~\ref{theo3x3three} about generators is proven in Section~\ref{section_3x3}. To prove the statement of  Theorem~\ref{theo3x3three} about separating set we describe certain elements of  $GL_3$-orbits on $\N_3^2$ in Section~\ref{section_Canon} (although we do not obtain classification of these orbits) and then complete the proof of key  Lemma~\ref{lemmaBig} in Sections~\ref{sectionCaseA}--\ref{sectionCaseDE}.

\subsection{Notations}
\label{section_notations}


Denote by $E$ the identity matrix and by $E_{ij}$ the matrix such that the $(i,j)^{\rm th}$ entry is equal to one and the rest of entries are zeros. We use the following notations for $3\times 3$ nilpotent matrices: $J_1=E_{12}$ and $J_2=E_{12}+E_{23}$.

Assume that $\algA$ is  $\mathcal{O}(M_n^d)^{GL_n}$ or $\mathcal{O}(\N_n^d)^{GL_n}$. We say that an $\NN$-homogeneous invariant $f\in \algA$ is {\it decomposable} and write $f\equiv0$ if $f$ is a polynomial in $\NN$-homogeneous invariants of $\algA$ of strictly lower degree. If $f$ is not decomposable, then we say that $f$ is {\it indecomposable} and write $f\not\equiv0$. In case $f-h\equiv0$ we write $f\equiv h$. Obviously, if $f\equiv0$ for $f\in \mathcal{O}(M_n^d)^{GL_n}$, then $\Psi(f)\equiv0$ in $\mathcal{O}(\N_n^d)^{GL_n}$.

\section{Cases of $2\times 2$ matrices and two $3\times 3$ matrices }\label{section_two_cases}

\begin{prop}\label{prop2x2}
Assume that $n=2$. Then the set 
$$S_{2,d}=\{\tr(Y_i Y_j), \; 1\leq i<j\leq d;\;\; \tr(Y_i Y_j Y_k), \; 1\leq i<j<k\leq d\}$$
is a minimal generating set as well as a minimal separating set for the algebra of invariants  $\mathcal{O}(\N_2^d)^{GL_2}$.  
\end{prop}
\begin{proof}   Applying the surjective homomorphism $\hat{\Psi}$ to the known  minimal generating set for $\mathcal{O}(M_2^d)^{GL_2}$ (see Section~\ref{section_separ}) we obtain that $S_{2,d}$ generates the algebra $\mathcal{O}(\N_2^d)^{GL_2}$. Then $S_{2,d}$ is also a separating set. Claims 1 and 2 (see below) show that $S_{2,d}$ is a {\it minimal} separating set  for all $d>0$, and therefore it is a minimal generating set. 

\medskip
\noindent{}{\it Claim 1.} Let $d=2$. Then $S_{2,2}\backslash \{\tr(Y_1 Y_2)\}$ is not separating for $\mathcal{O}(\N_2^2)^{GL_2}$.
\smallskip

To prove this claim consider $\un{A}=(E_{12},E_{12})$ and $\un{B}=(E_{12},E_{21})$ from $\N_2^2$. Then $\tr(A_1 A_2)\neq \tr(B_1 B_2)$.

\medskip
\noindent{}{\it Claim 2.} Let $d=3$. Then $S_{2,3}\backslash \{\tr(Y_1 Y_2 Y_3)\}$ is not separating for $\mathcal{O}(\N_2^3)^{GL_2}$.
\smallskip

To prove this claim we consider $\un{A}=(E_{12},-E_{21}, C)$ and $\un{B}=(E_{12},C,-E_{21})$ from $\N_2^3$, where $C=E_{11}+E_{12}-E_{21}-E_{22}$. Then $\tr(A_i A_j)=\tr(B_i B_j)$ for all $1\leq i<j\leq 3$, but $\tr(A_1 A_2 A_3)\neq \tr(B_1 B_2 B_3)$.
\end{proof}

\begin{prop}\label{prop3x3two}
Assume that $n=3$ and $d=2$. Then the set 
$$S_{3,2}=\{\tr(Y_1 Y_2), \;\tr(Y_1^2 Y_2), \;\tr(Y_1 Y_2^2), \;\tr(Y_1^2 Y_2^2), \;\tr(Y_1^2 Y_2^2Y_1 Y_2)\}$$
is a minimal generating set as well as a minimal separating set for the algebra of invariants  $\mathcal{O}(\N_3^2)^{GL_3}$.  
\end{prop}
\begin{proof} Applying the surjective homomorphism $\hat{\Psi}$ to the known  minimal generating set for $\mathcal{O}(M_3^2)^{GL_3}$ (see Theorem~1 of~\cite{Lopatin_Sib}) we obtain that $S_{3,2}$ generates the algebra $\mathcal{O}(\N_3^2)^{GL_3}$. Then $S_{3,2}$ is also a separating set. To complete the proof it is enough to show that $S_{3,2}$ is a {\it minimal} separating set, i.e., for any $f\in S_{3,2}$ it is enough to show that $S_{3,2}\backslash \{f\}$ is not separating. We will construct $\un{A}=(A_1,A_2)$ and $\un{B}=(B_1,B_2)$ from $\N_3^2$ such that for any $h\in S_{3,2}\backslash \{f\}$ we have $h(\un{A})=h(\un{B})$, but $f(\un{A})\neq f(\un{B})$.

For $f=\tr(Y_1 Y_2)$, we consider $A_1=B_1=J_2$, $A_2=E_{32}$, and $B_2=E_{12}$.

For $f=\tr(Y_1^2 Y_2)$, we consider $A_1=B_1=J_2$, 
$$A_2=\left(
\begin{array}{ccc}
0 & 1 & 0 \\
1 & 0 & -1 \\
0 & 1 & 0 \\
\end{array}
\right),\;\text{ and }\;
B_2=\left(
\begin{array}{ccc}
0 & 0 & 0 \\
1 & 0 & 0 \\
-1 & 1 & 0 \\
\end{array}
\right).
$$

For $f=\tr(Y_1^2 Y_2^2)$, we consider $A_1=B_1=J_2$,
$$A_2=\left(
\begin{array}{ccc}
0 & -2 & 1 \\
2 & 0 & 1 \\
2 & 2 & 0 \\
\end{array}
\right),\;\text{ and }\;
B_2=\left(
\begin{array}{ccc}
0 & 0 & 2 \\
0 & 0 & -1 \\
2 & 4 & 0 \\
\end{array}
\right).
$$

For $f=\tr(Y_1^2 Y_2^2 Y_1 Y_2)$, we consider $A_1=B_1=J_2$,
$$A_2=\left(
\begin{array}{ccc}
0 & 1 & 0 \\
0 & 0 & 0 \\
1 & 1 & 0 \\
\end{array}
\right),\;\text{ and }\;
B_2=\left(
\begin{array}{ccc}
0 & 0 & -1 \\
0 & 0 & 1 \\
1 & 1 & 0 \\
\end{array}
\right).
$$
\end{proof}

\begin{cor}\label{cor} Assume that $d\geq2$. Then
\begin{enumerate}
\item[(a)]  for the algebra $\mathcal{O}(M_3^d)^{GL_3}$ we have $D_{\rm sep} = 6$, where the field $\FF$ is arbitrary of characteristic different from $2$ and $3$;

\item[(b)] for the algebra $\mathcal{O}(\N_3^d)^{GL_3}$ we have $D_{\rm gen} = D_{\rm sep} = 6$.
\end{enumerate}
\end{cor}
\begin{proof} Since it is well-known that $\mathcal{O}(M_3^d)^{GL_3}$ is generated by homogeneous invariants of degree less or equal to $6$ over an arbitrary field of characteristic different from $2$ and $3$ (for example, see~\cite{Lopatin_Comm1}), we have $D_{\rm sep}(\mathcal{O}(M_3^d)^{GL_3}) \leq 6$ and $D_{\rm sep}(\mathcal{O}(\N_3^d)^{GL_3}) \leq D_{\rm gen}(\mathcal{O}(\N_3^d)^{GL_3}) \leq 6$. Proposition~\ref{prop3x3two} shows that any homogeneous invariant of $\mathcal{O}(\N_3^2)^{GL_3}$ of degree less than $6$ belongs to the subalgebra generated by $S_{3,2}\backslash\{\tr(Y_1^2 Y_2^2Y_1 Y_2)\}$. Thus, $D_{\rm sep}(\mathcal{O}(\N_3^2)^{GL_3})>5$ and the proof of part (b) is completed.  

Assume that $\FF$ is an arbitrary field of characteristic different from $2$ and $3$. Then the set
$$S'_{3,2} = S_{3,2}|_{Y_1\to X_1, Y_2\to X_2} \bigcup\; \{\tr(X_i^k)\,|\,i=1,2,\; k=1,2,3\}$$
is known to be a minimal generating set for $\mathcal{O}(M_3^d)^{GL_3}$ (for example, see~\cite{Lopatin_Comm1}) and repeating the proof of Proposition~\ref{prop3x3two} we can see that $S'_{3,2}$ is a minimal separating set for $\mathcal{O}(M_3^d)^{GL_3}$. The proof of part (a) is concluded similarly to part (b).
 \end{proof}

\section{The case of three $3\times 3$ nilpotent matrices}\label{section_3x3}

Denote by $S_{3,3}$ the following subset of $\mathcal{O}(\N_3^3)^{GL_3}$:
$$\tr(Y_i Y_j), \;\tr(Y_i^2 Y_j), \;\tr(Y_i Y_j^2), \;\tr(Y_i^2 Y_j^2), \;\tr(Y_i^2 Y_j^2 Y_i Y_j) \text{ with } 1\leq i<j\leq 3,$$
$$\tr(Y_1 Y_2 Y_3),\; \tr(Y_1 Y_3 Y_2)$$
$$ \tr(Y_i^2 Y_j Y_k)\text{ with } \{i,j,k\}=\{1,2,3\},$$
$$ \tr(Y_1^2 Y_2 Y_1 Y_3),\; \tr(Y_2^2 Y_1 Y_2 Y_3),\; \tr(Y_3^2 Y_1 Y_3 Y_2).$$
We set $P_{3,3}=S_{3,3}\sqcup P'_{3,3}$, where $P'_{3,3}\subset \mathcal{O}(\N_3^3)^{GL_3}$ is the following set:
$$\tr(Y_i^2 Y_j^2 Y_k),\; \tr(Y_i^2 Y_j^2 Y_i Y_k) \text{ with } \{i,j,k\}=\{1,2,3\},
\; \tr(Y_1^2 Y_2^2 Y_3^2).$$
Note that $|S_{3,3}|=26$ and $|P_{3,3}|=39$.

\begin{remark}\label{remarkP33}
Applying the surjective homomorphism $\hat{\Psi}$ to the  minimal generating set of $\mathcal{O}(M_3^3)^{GL_3}$ from Theorem~1 of~\cite{Lopatin_Sib} we obtain that $P_{3,3}$ generates the algebra $\mathcal{O}(\N_3^3)^{GL_3}$.
\end{remark}

\begin{theo}\label{theo3x3three}
Assume that $n=d=3$ and the field $\FF$ is algebraically closed of characteristic zero. Then 
\begin{enumerate}
\item[(a)] $P_{3,3}$ is a  minimal generating set for the algebra of invariants $\mathcal{O}(\N_3^3)^{GL_3}$. 

\item[(b)] $S_{3,3}$ is a  minimal separating set for the algebra of invariants  $\mathcal{O}(\N_3^3)^{GL_3}$;
\end{enumerate}
\end{theo}
\begin{proof}
We split the proof into several lemmas. Remark~\ref{remarkP33} and   Lemmas~\ref{lemmaSepMin} and ~\ref{lemmaInd} (see below) imply that  $P_{3,3}=S_{3,3}\sqcup P'_{3,3}$ satisfies conditions (a) and (b) of Remark~\ref{remarkMinGen} (see below), which implies that $P_{3,3}$ is a {\it minimal} generating set for $\mathcal{O}(\N_3^3)^{GL_3}$.    

To complete the proof it is enough to show that $S_{3,3}$ is a  separating set for  $\mathcal{O}(\N_3^3)^{GL_3}$, since the minimality follows from Lemma~\ref{lemmaSepMin}. Thus the fact that $P_{3,3}$ is separating together with Lemma~\ref{lemmaBig} (see below) completes the proof.
\end{proof}

\begin{remark}\label{remarkMinGen}  Let $\algA$ be  $\mathcal{O}(M_n^d)^{GL_n}$ or $\mathcal{O}(\N_n^d)^{GL_n}$. Assume that a homogeneous set $P=P_1\cup P_2$ generates the algebra $\algA$ and $P_1, P_2\subset \algA$ satisfy  the following conditions:
\begin{enumerate}
\item[(a)] $P\backslash \{f\}$ is not a separating set for $\algA$ for any $f\in P_1$;

\item[(b)] any non-trivial linear combination of elements $f_1,\ldots,f_m\in P_2$ of the same multidegree is not decomposable in  $\algA$.
\end{enumerate}
Then $P$ is a minimal generating set for  $\algA$.
\end{remark}
\begin{proof} If $P$ is not a minimal generating set for  $\algA$, then there exists an $f\in P$ such that $P\backslash \{f\}$ is also a generating set for $\algA$. But condition (a) implies that $f$ does not belong to $P_1$ and  condition (b) implies that $f$ does not belong to $P_2$; a contradiction. 
\end{proof}

\begin{lemma}\label{lemmaSepMin}  For any $f\in S_{3,3}$ the set $P_{3,3}\backslash \{f\}$ is not separating for $\mathcal{O}(\N_3^3)^{GL_3}$. In particular, any proper subset of $S_{3,3}$ is not a separating set for $\mathcal{O}(\N_3^3)^{GL_3}$. 
\end{lemma}
\begin{proof} If $f$ depends on entries of only two matrices from the list $Y_1,Y_2,Y_3$, then the claim of the lemma follows from the fact that $S_{3,2}$ is a minimal separating set for $\mathcal{O}(\N_3^2)^{GL_3}$ (see Proposition~\ref{prop3x3two}). 

We will construct $\un{A}=(A_1,A_2,A_3)$ and $\un{B}=(B_1,B_2,B_3)$ from $\N_3^3$ such that for any $h\in P_{3,3}\backslash \{f\}$ we have $h(\un{A})=h(\un{B})$, but $f(\un{A})\neq f(\un{B})$.

For $f=\tr(Y_1 Y_2 Y_3)$ we consider $\un{A}=(E_{31}, E_{12}+E_{32}, E_{23})$ and  $\un{B}=(0,E_{12}, E_{21})$.

For $f=\tr(Y_1^2 Y_2 Y_3)$ we consider $\un{A}=(E_{21}+E_{32}, E_{12}, E_{23})$ and $\un{B}= (E_{13}+E_{21}, E_{12}, J_2)$.

For $f=\tr(Y_1^2 Y_2 Y_1 Y_3)$ we consider $\un{A}=(E_{21}+E_{32},E_{13},E_{23})$ and $\un{B}=(E_{23}+E_{31},E_{12},E_{13})$. Obviously, the considered cases of $f\in S_{3,3}$ imply the rest of the cases.
\end{proof}

\begin{lemma}\label{lemmaInd} The following elements are indecomposable in  $\mathcal{O}(\N_3^3)^{GL_3}$:
\begin{enumerate}
\item[(a)]$\al_1 \tr(Y_1^2 Y_2^2 Y_3) + \al_2\tr(Y_2^2 Y_1^2 Y_3)$ for any $\al_1,\al_2\in\FF$, where $\al_1$ or $\al_2$ is non-zero;

\item[(b)]$\tr(Y_1^2 Y_2^2 Y_1 Y_3)$; 

\item[(c)]$\tr(Y_1^2 Y_2^2 Y_3^2)$. 
\end{enumerate}
\end{lemma}
\begin{proof} 
For short, denote
$$ C = \left(
\begin{array}{ccc}
0 & -1 & -1 \\
0 & 1  & 1  \\
1 & 0  & -1 \\
\end{array}
\right)\;\text{ and }\;
D = \left(
\begin{array}{ccc}
0 & 0 & 0 \\
0 & 1 & 1  \\
0 & -1 & -1 \\
\end{array}
\right).$$

\noindent{\bf (a)} Assume that the claim of part (a) does not hold. Since the algebra $\mathcal{O}(\N_3^3)^{GL_3}$ has $\NN^d$-grading by multidegrees and is known to be generated by $P_{3,3}$ (see Remark~\ref{remarkP33}), without loss of generality we obtain that there exist $\al_1\neq0,\al_2,\be_1,\ldots,\be_4$ from $\FF$ such that 
$$\begin{array}{rcl}
\al_1 \tr(A_1^2 A_2^2 A_3) + \al_2\tr(A_2^2 A_1^2 A_3)\!\!\!\! & = &\!\!\!\! 
\be_1 \tr(A_1 A_2 A_3)  \tr(A_1 A_2) + \be_2 \tr(A_1 A_3 A_2) \tr(A_1 A_2) \\
&+&\!\!\!\! \be_3 \tr(A_1^2 A_2) \tr(A_2 A_3) + \be_4 \tr(A_2^2 A_1) \tr(A_1 A_3) \\
\end{array}
$$
for all $A_1,A_2,A_3$ from $\N_3$. Consequently considering triples $(J_2,E_{23} + E_{31}, E_{13})$ and $(J_2, C, E_{21})$ we obtain $\be_3=0$ and $\al_1=0$, respectively; a contradiction.

\medskip
\noindent{\bf (b)}  Assume that the claim of part (b) does not hold. As in part~(a), we can see that there are $\al_1, \al_2, \al_3, \be_1, \be_2, \be_3, \ga\in\FF$ such that 
$$
\begin{array}{rcl}
\tr(A_1^2 A_2^2 A_1 A_3)\!\!\!\! & = &\!\!\!\!\al_1\! \tr(A_1^2 A_2^2)\! \tr(A_1 A_3) \!+\! \al_2\! \tr(A_1^2 A_2 A_3)\! \tr(A_1 A_2) \!+\! \al_3\! \tr(A_1^2 A_3 A_2)\! \tr(A_1 A_2) \\
& + &\!\!\!\!  \be_1\! \tr(A_1^2 A_2)\! \tr(A_1 A_2 A_3) \!+\! \be_2\! \tr(A_1^2 A_2)\! \tr(A_1 A_3 A_2) \!+\! \be_3\! \tr(A_2^2 A_1)\! \tr(A_1^2 A_3) \\ 
&+ &\!\!\!\!  \ga \tr(A_1 A_2)^2 \tr(A_1 A_3)
\end{array}
$$
for all triples $(A_1,A_2,A_3)\in\N_3^3$. Consequently considering triples $(J_2, E_{31} + E_{32}, E_{21} - E_{32})$, $(J_2, E_{31} + E_{32},  E_{32})$, $(J_2, C, E_{32})$, $(J_2, C, E_{31}+ E_{32})$, $(J_2,C,E_{21}-E_{32})$ we obtain that $\al_2=0$, $\ga=0$, $\be_2=\al_1+\be_1$, $\al_1=1$, $\be_1=0$, respectively. Then we take 
$$ A_1 = \left(
\begin{array}{ccc}
1 & 1 & 0 \\
-1 & -1  & 1  \\
0 & 0  & 0 \\
\end{array}
\right), \;
A_2 = \left(
\begin{array}{ccc}
0 & -1 & 1 \\
1 & 0  & 0  \\
1 & 0  & 0 \\
\end{array}
\right),
$$
and $A_3=E_{12}$ to obtain $0=-1$; a contradiction. 

\medskip
\noindent{\bf (c)} Assume that the claim of part (c) does not hold. As in part~(a), we can see that there are $\al_1, \ldots, \al_6, \be_1, \ldots, \be_6, \ga\in\FF$ such that 
$$ \begin{array}{rcl}
\tr(A_1^2 A_2^2 A_3^2)\!\!\!\! & = & \!\!\!\!\al_1\!\tr(A_1^2 A_2 A_3)\!\tr(A_2 A_3)\!+\!\al_2\!\tr(A_1^2 A_3 A_2)\!\tr(A_2 A_3)\!+\!\al_3\!\tr(A_2^2 A_1 
A_3)\!\tr(A_1 A_3)\\ 
&+&\!\!\!\!\al_4\!\tr(A_2^2 A_3 A_1)\!\tr(A_1 A_3)\!+\!\al_5\!\tr(A_3^2 A_1 
A_2)\!\tr(A_1 A_2)\!+\!\al_6\!\tr(A_3^2 A_2 A_1)\!\tr(A_1 A_2) \\
& + &\!\!\!\!\be_1\!\tr(A_1 A_2 A_3)^2  + \be_2\!\tr(A_1 A_2 A_3) \tr(A_1 A_3 A_2) + \be_3\!\tr(A_1 A_3 A_2)^2 \\
& + &\!\!\!\!\be_4\!\tr(A_1^2 A_2) \tr(A_3^2 A_2) + \be_5\!\tr(A_1^2 A_3) \tr(A_2^2 A_3) + \be_6\!\tr(A_2^2 A_1) \tr(A_3^2 A_1) \\
& + &\!\!\!\!\ga\!\tr(A_1 A_2) \tr(A_1 A_3) \tr(A_2 A_3)\\
\end{array}
$$
for all triples $(A_1,A_2,A_3)\in\N_3^3$.  Consequently considering triples $(J_2, E_{23} + E_{31}, E_{31})$, $(J_2, E_{31}, E_{23} + E_{31})$, and $(J_2, D, E_{31}+E_{32})$ we obtain that $\be_1=0$, $\be_3=0$, and $\be_2=-\ga$, respectively.  Considering all permutations of entries in the triple $(J_2,D+E_{31},E_{32})$ we can see that $\al_1=\cdots=\al_6=-\ga$. Similarly, taking all permutations of entries in the triple $(J_2,J_2,E_{23}+E_{31})$ we obtain that $\be_4=\be_5=\be_6=\ga$. Considering triple $(J_2, C, E_{21}+E_{32})$ we can see see that $\ga=0$. Finally, taking $A_2=J_2$, $A_3=E_{21}+E_{32}$, and 
$$A_1 = \left(
\begin{array}{ccc}
0 & 0 & -1 \\
0 & 0 & 1  \\
1 & 1 & 0 \\
\end{array}
\right)$$
we obtain $0=-1$; a contradiction.
\end{proof}

\begin{lemma}\label{lemmaBig}
Consider $\un{A}=(A_1,A_2,A_3)$ and $\un{B}=(B_1,B_2,B_3)$ from $\N_3^3$ that are not separated by $S_{3,3}$.  Then $\un{A}$, $\un{B}$ are not separated by $P_{3,3}$.
\end{lemma}
\begin{proof}
By the definition of $P_{3,3}'$ we can assume that 
\begin{eq}\label{eqCond0}
A_i\neq 0 \text{ or }B_i\neq 0 \text{ for all } 1\leq i\leq 3.
\end{eq}%
Since $\FF$ is algebraically closed, then applying the action of $GL_3$ on the pair $(\un{A}, \un{B})$ we can assume that we have one of the following cases:
\begin{enumerate}
\item[(a)] $A_1= J_2$ and $B_1=J_2$;

\item[(b)] $A_1= J_2$ and $B_1=J_1$;

\item[(c)] $A_1= J_1$ and $B_1=J_1$;

\item[(d)] $A_1= J_1$ and $B_1=0$;

\item[(e)] $A_1= J_2$ and $B_1=0$.
\end{enumerate}
Here we also applied renaming of $\un{A}$ and $\un{B}$ to diminish the number of cases. Lemmas~\ref{lemmaA}, \ref{lemmaB}, \ref{lemmaC}, \ref{lemmaD}, \ref{lemmaE} (see below) conclude the proof.
\end{proof}

\section{Canonical forms}\label{section_Canon}

\begin{lemma}\label{lemmaCan1} Assume $A_1=J_1$ and $A_2\in \N_3$. Then there exists $g\in GL_3$ such that $g A_1 g^{-1} = A_1$ and  $g A_2 g^{-1}$ is one of the following matrices:
\begin{enumerate}
\item[(1)] $V_{\rm I}=\left(
\begin{array}{ccc}
0 & a_2 & a_3 \\
0 & 0 & a_6 \\
0 & 0 & 0 \\
\end{array}
\right); $

\item[(2)] $V_{\rm II}=\left(
\begin{array}{ccc}
0 & 0 & a_3 \\
0 & a_5 & -a_5^2 \\
0 & 1 & -a_5 \\
\end{array}
\right); $

\item[(3)] $V_{\rm III}=\left(
\begin{array}{ccc}
a_1 & a_2 & -a_1^2 \\
0 & -a_1 & a_6 \\
1 & 0 & 0 \\
\end{array}
\right); $

\item[(4)]  $V_{\rm IV}=\left(
\begin{array}{ccc}
a_1 & a_2 & a_3 \\
a_4 & 0 & 0 \\
a_7 & a_8 & -a_1 \\
\end{array}
\right) \mbox{ with }a_4\neq 0. $
\end{enumerate}
\end{lemma}
\begin{proof}
\medskip{}
\noindent{\bf(1)} Assume that $A_2$ is an upper triangular matrix. The condition $A_2\in \N_3$ concludes the proof of part (1).
\medskip{}

In the rest of the proof we consider the following $g\in GL_3$ with $gA_1g^{-1}=A_1$, where $g_1,g_9\neq 0$:  
$$g=\left(
\begin{array}{ccc}
g_1 & g_2 & g_3 \\
0 & g_1 & 0 \\
0 & 0 & g_9 \\
\end{array}
\right) \mbox{ and } 
A_2=\left(
\begin{array}{ccc}
a_1 & a_2 & a_3 \\
a_4 & a_5 & a_6 \\
a_7 & a_8 & a_9 \\
\end{array}
\right).
$$ 

\medskip{}
\noindent{\bf(2)} Assume that $a_4=a_7=0$ and $a_8\neq0$. Taking $g_1=1$, $g_3=\frac{1}{a_8}((a_1-a_5)g_2 - a_2)$ and $g_9=\frac{1}{a_8}$ we obtain 
$$gA_2g^{-1}=\left(
\begin{array}{ccc}
	\ast & 0 & \ast \\
	0 & \ast & \ast \\
	0 & 1 & \ast \\
\end{array}
\right). $$
Then the claim of part (2) follows from the condition $A_2\in \N_3$.
\medskip{}

\medskip{}
\noindent{\bf(3)} Assume that  $a_4=0$ and $a_7\neq 0$. Taking  $g_2=\frac{g_1a_8}{a_7}$, $g_3= \frac{g_1a_9}{a_7}$ and $g_9= \frac{g_1}{a_7}$ we obtain 
$$gA_2g^{-1}=\left(
\begin{array}{ccc}
	\ast & \ast & \ast \\
	0 & \ast & \ast \\
	1 & 0 & 0 \\
\end{array}
\right). $$
The condition $A_2\in \N_3$ concludes the proof of part (3).

\medskip{}
\noindent{\bf(4)} Assume that  $a_4\neq 0$. Taking  $g_2=\frac{g_1a_5}{a_4}$ and $g_3= \frac{a_6g_1}{a_4}$ we obtain  
$$gA_2g^{-1}=\left(
\begin{array}{ccc}
\ast & \ast & \ast \\
a_4 & 0 & 0 \\
\ast & \ast & \ast \\
\end{array}
\right). $$
The condition $A_2\in \N_3$ concludes the proof of part (4).
\end{proof}

\begin{lemma}\label{lemmaCan2} Assume $A_1=J_2$ and $A_2\in \N_3$. Then there exists $g\in GL_3$ such that $g A_1 g^{-1} = A_1$ and  $g A_2 g^{-1}$ is one of the following matrices:
\begin{enumerate}
\item[(1)] $W_{\rm I}=\left(
\begin{array}{ccc}
0 & a_2 & a_3 \\
0 & 0 & a_6 \\
0 & 0 & 0 \\
\end{array}
\right); $
\item[(2)] $W_{\rm II}=\left(
\begin{array}{ccc}
0 & 0 & a_3 \\
0 & 0 & 0 \\
0 & a_8 & 0\\
\end{array}
\right) \mbox{ with }a_8\neq 0; $ 
\item[(3)] $W_{\rm III}=\left(
\begin{array}{ccc}
0 & 0 & a_3 \\
a_4 & 0 & 0 \\
0 & 0 & 0\\
\end{array}
\right) \mbox{ with }a_4\neq 0; $
\item[(4)] $W_{\rm IV}=\left(
\begin{array}{ccc}
0 & 0 & 0 \\
a_4 & a_5 & a_6 \\
0 & a_8 & -a_5\\
\end{array}
\right) \mbox{ with }a_4,a_8\neq 0; $
\item[(5)] $W_{\rm V}=\left(
\begin{array}{ccc}
0 & a_2 & a_3 \\
0 & a_5 & a_6 \\
a_7 & a_8 & -a_5\\
\end{array}
\right) \mbox{ with }a_7\neq 0. $
\end{enumerate}
\end{lemma}
\begin{proof}
\noindent{\bf(1)} Assume that $A_2$ is an upper triangular matrix. Since $A_2\in \N_3$, the diagonal elements are zeros. Then taking $g=E$ we conclude the proof of part (1).
\medskip{}

In the rest of the proof we consider the following $g\in GL_3$ with $gA_1g^{-1}=A_1$, where $g_1\neq 0$:  
$$g=\left(
\begin{array}{ccc}
g_1 & g_2 & g_3 \\
0 & g_1 & g_2 \\
0 & 0 & g_1 \\
\end{array}
\right) \mbox{ and } 
A_2=\left(
\begin{array}{ccc}
a_1 & a_2 & a_3 \\
a_4 & a_5 & a_6 \\
a_7 & a_8 & a_9 \\
\end{array}
\right).
$$

\medskip{}
\noindent{\bf(2)} Assume that $a_4=a_7=0$ and $a_8\neq 0$. Taking $g_2= \frac{-g_1a_5}{a_8}$ and $g_3=-(a_2+\frac{a_1a_5-a_5^2}{a_8})\frac{g_1}{a_8}$ we obtain $$gA_2g^{-1}=\left(
\begin{array}{ccc}
\ast & 0 & \ast \\
0 & 0 & \ast \\
0 & a_8 & \ast \\
\end{array}
\right). $$
The condition $A_2\in \N_3$ concludes the proof of part (2).

\medskip{}
\noindent{\bf(3)} Assume that $a_7=a_8=0$ and $a_4\neq 0$. Taking $g_2= \frac{-g_1a_1}{a_4}$ and $g_3=\frac{g_1}{a_4^2}(a_1^2+a_4a_6+a_1(a_5-a_9))$ we obtain
$$gA_2g^{-1}=\left(
\begin{array}{ccc}
0 & \ast & \ast \\
a_4 & \ast & 0 \\
0 & 0 & \ast  \\
\end{array}
\right). $$
The condition $A_2\in \N_3$ concludes the proof of part (3).

\medskip{}
\noindent{\bf(4)} Assume that $a_7=0$ and $a_4,a_8\neq 0$. Taking $g_2= \frac{-g_1a_1}{a_4}$ and $g_3=(\frac{a_1a_5}{a_4}-a_2)\frac{g_1}{a_8}$ we obtain
$$gA_2g^{-1}=\left(
\begin{array}{ccc}
0 & 0 & \ast \\
a_4 & \ast & \ast \\
0 & a_8 & \ast \\
\end{array}
\right). $$ The condition $A_2\in \N_3$ concludes the proof of part (4).

\medskip{}
\noindent{\bf(5)} Assume that $a_7\neq 0$. Taking $g_2= \frac{-g_1a_4}{a_7}$ and $g_3=(\frac{a_4^2}{a_7}-a_1)\frac{g_1}{a_7}$ we obtain 
$$gA_2g^{-1}=\left(
\begin{array}{ccc}
0 & \ast & \ast \\
0 & \ast & \ast \\
a_7 & \ast & \ast \\
\end{array}
\right). $$ 
The condition $A_2\in \N_3$ concludes the proof of part (5).
\end{proof}

\section{Case (a): $A_1= J_2$ and $B_1=J_2$.}\label{sectionCaseA}


In Lemma~\ref{lemmaA}, \ref{lemmaB}, \ref{lemmaC}, \ref{lemmaD}, \ref{lemmaE} (see below) we will consider  $\un{A}=(A_1,A_2,A_3)$ and $\un{B}=(B_1,B_2,B_3)$ from $\N_3^3$ with the certain matrices $A_1$ and $B_1$. Denote the entries of matrices $A_i,B_i$, where  $i=2,3$:
$$A_i=\left(\begin{array}{ccc}
a_{i1} & a_{i2} & a_{i3} \\
a_{i4} & a_{i5} & a_{i6} \\
a_{i7} & a_{i8} & a_{i9} \\
\end{array}
\right) \;\text{ and }\;
B_i=\left(\begin{array}{ccc}
b_{i1} & b_{i2} & b_{i3} \\
b_{i4} & b_{i5} & b_{i6} \\
b_{i7} & b_{i8} & b_{i9} \\
\end{array}
\right). $$
For short, denote by $T_{i_1,\ldots,i_k}$ the equality $f(\un{A})=f(\un{B})$ for $f=\tr(Y_{i_1}\cdots Y_{i_k})$. We say that we can express the variable $c\in\{a_{ij}, b_{ij}\,|\,i=2,3,\, 1\leq j\leq 9\}$ from  $T_{i_1,\ldots,i_k}$ if  $T_{i_1,\ldots,i_k}$ can be rewritten as $c f = h$, where polynomials $f,h$ in commutative variables $\{a_{ij}, b_{ij}\}$ do not contain $c$ and $f\neq 0$. 

We say that a matrix $A\in\N_3$ has type $V_{\rm I},\ldots,V_{\rm IV}, W_{\rm I},\ldots, W_{\rm V}$, respectively, if $A$ is equal to the corresponding matrix from Lemma~\ref{lemmaCan1} or Lemma~\ref{lemmaCan2}. In the proofs of Lemmas~\ref{lemmaA}, \ref{lemmaB}, \ref{lemmaC} (see below) we say that the type of the pair $(A_2,B_2)$ is  $(K,R)$ for some symbols $K,R$ from the set of symbols $\{V_{\rm I},\ldots,V_{\rm IV}, W_{\rm I},\ldots, W_{\rm V}\}$ if 
\begin{enumerate}
\item[$\bullet$] the type of $A_2$ is $K$ and the elements of $A_2$ are the result of substitutions $a_i\to a_{2i}$, where $1\leq i\leq 9$, in the corresponding matrix from Lemma~\ref{lemmaCan1} or Lemma~\ref{lemmaCan2}. 

\item[$\bullet$] the type of $B_2$ is $L$ and the elements of $B_2$ are the result of substitutions $a_i\to b_{2i}$, where $1\leq i\leq 9$, in the corresponding matrix from Lemma~\ref{lemmaCan1} or Lemma~\ref{lemmaCan2}. 
\end{enumerate}

\begin{lemma}\label{lemmaA}
Consider $\un{A}=(J_2,A_2,A_3)$ and $\un{B}=(J_2,B_2,B_3)$ from $\N_3^3$ that are not separated by $S_{3,3}$.  Then $\un{A}$, $\un{B}$ are not separated by $P_{3,3}$.
\end{lemma}
\begin{proof}
 By Lemma~\ref{lemmaCan2}, we can assume that $A_2$ and $B_2$ have one of the following types: $W_{\rm I},\ldots,W_{\rm V}$.  The equality  $T_{112}$ implies $b_{27}=a_{27}$. 

Assume $a_{27}\neq 0$. Then the pair $(A_2,B_2)$ has type $(W_{\rm V},W_{\rm V})$.  Consequently considering $T_{12}$ and $T_{113}$ we obtain $b_{28}=a_{28}$ and $b_{37}=a_{37}$, respectively. Since $a_{27}\neq 0$, consequently considering $T_{1122}$, $T_{112212}$, $T_{1132}$ and $T_{11213}$ we obtain $b_{25} = a_{25}$, $b_{22} = a_{22}$, $b_{39} = a_{39}$ and $b_{34} = a_{34}$, respectively. Continuing the calculations, let us consequently consider $T_{13}$,  $T_{221}$,  $T_{1123}$,  $\tr(A_3)-\tr(B_3)=0$, $\sigma_2(A_2)-\sigma_2(B_2)=0$, $T_{123}$,  $T_{132}$,  $T_{23}$ to conclude that $b_{38} = a_{38}$, $b_{26} = a_{26}$, $b_{31} = a_{31}$,  $b_{35} = a_{35}$, $b_{23} = a_{23}$, $b_{32} = a_{32}$, $b_{36} = a_{36}$ and $b_{33} = a_{33}$, respectively. Therefore, $\un{A}=\un{B}$ are not separated by $P_{3,3}$.

Assume $a_{27}=0$. Then the equality $T_{12}$ implies that  $b_{28}=a_{24}+a_{28}-b_{24}$. Considering $T_{1122}$  we conclude that $a_{24}-b_{24}=0$ or $a_{28}-b_{24}=0$. We divide the proof into two cases.

\medskip
\noindent{\bf 1.} Let $b_{24}\neq a_{24}$. Hence $b_{24}=a_{28}$ and $b_{28}=a_{24}$. These restrictions show that the only possibilities for the type of the pair $(A_2,B_2)$ are $(W_{\rm II},W_{\rm III})$, $(W_{\rm III},W_{\rm II})$ and $(W_{\rm IV},W_{\rm IV})$. Obviously, it is enough to consider the first and the last cases.

\medskip
\noindent{\bf 1.1.} Assume that the type of  $(A_2,B_2)$ is $(W_{\rm II},W_{\rm III})$. Since $a_{28}\neq0$, consequently considering $T_{113}$, $T_{1123}$, $T_{1132}$, $T_{13}$ and $T_{11213}$ we obtain $b_{37}=a_{37}$, $a_{34}=0$, $b_{38}=0$, $b_{34}=a_{38}$ and $a_{37}=0$, respectively. Thus $\un{A}$, $\un{B}$ are not separated by $P_{3,3}\backslash S_{3,3}$ and the required is proven. 

\medskip
\noindent{\bf 1.2.} Assume now that the type of $(A_2,B_2)$ is $(W_{\rm IV},W_{\rm IV})$. Then $a_{24}\neq 0$ and $a_{28}\neq0$. The equality $T_{113}$ implies that $b_{37}=a_{37}$ and $T_{112212}$ implies that $a_{24} a_{28} (a_{24}-a_{28})=0$; a contradiction.

\bigskip
\noindent{\bf 2.} Let $b_{24}=a_{24}$. Hence $b_{28}=a_{28}$. It is easy to see that the are only four possibilities for the type of the pair $(A_2,B_2)$, which are $(W_{\rm I},W_{\rm I})$, $(W_{\rm II},W_{\rm II})$, $(W_{\rm III},W_{\rm III})$  and $(W_{\rm IV},W_{\rm IV})$. 

\medskip
\noindent{\bf 2.1.} Assume that the type of $(A_2,B_2)$ is $(W_{\rm I},W_{\rm I})$. Consequently considering equalities $\tr(A_3)=0$, $\tr(B_3)=0$, $T_{13}$, $T_{113}$ we obtain $a_{39}=-a_{31}-a_{35}$, $b_{39}=-b_{31}-b_{35}$,  $b_{38}=a_{34}+a_{38}-b_{34}$, $b_{37}=a_{37}$, respectively. 

Let $a_{37}\neq 0$. Consequently considering equalities $T_{123}$, $T_{132}$, $T_{1133}$, $T_{23}$,  we obtain $b_{26}=a_{26}$, $b_{22}=a_{22}$, $a_{35}=((a_{34}-b_{34})(a_{38}-b_{34})+a_{37}b_{35})/a_{37}$, $b_{23}=((a_{22}-a_{26})(a_{34}-b_{34})+a_{23}a_{37})/a_{37}$, respectively.
From $T_{113313}$, $T_{331}$ we can express $b_{32}$ and $b_{36}$, respectively. Hence $\un{A}$, $\un{B}$ are not separated by $P_{3,3}$ and the required is proven. 

Let $a_{37}=0$ and $a_{34}-b_{34}\neq 0$. Then the equality $T_{1133}$ implies  $b_{34}=a_{38}$. If $a_{34}=0$, then we can see that $\un{A}$, $\un{B}$ are not separated by $P_{3,3}\backslash S_{3,3}$. On the other hand,  if $a_{34}\neq 0$, then $T_{113313}$ implies that $a_{38}=0$ and we also obtain that $\un{A}$, $\un{B}$ are not separated by $P_{3,3}\backslash S_{3,3}$. 

Finally, let $a_{37}=0$ and $a_{34}-b_{34}=0$. If $a_{34}=0$ or $a_{38}=0$, then  $\un{A}$, $\un{B}$ are not separated by $P_{3,3}\backslash S_{3,3}$. On the other hand, if $a_{34}$ and $a_{38}$ are not equal to zero, then considering $T_{3312}$ and $T_{3321}$ we obtain $b_{26}=a_{26}$ and $b_{22}=a_{22}$, respectively.  Thus $\un{A}$, $\un{B}$ are not separated by $P_{3,3}\backslash S_{3,3}$. 

\medskip
\noindent{\bf 2.2.} Assume that the type of $(A_2,B_2)$ is $(W_{\rm II},W_{\rm II})$. Since $a_{28}\neq0$, consequently considering equalities $T_{113}$, $T_{123}$, $T_{132}$, $\tr(A_3)-\tr(B_3)=0$, $T_{1123}$, $T_{13}$ we obtain that $b_{37}=a_{37}$, $b_{35}=a_{35}$, $b_{39}=a_{39}$, $b_{31}=a_{31}$, $b_{34}=a_{34}$, $b_{38}=a_{38}$, respectively. 

In case $a_{37}\neq0$ we apply consequently $T_{113313}$, $T_{331}$, $T_{23}$ to see that $b_{32}=a_{32}$, $b_{36}=a_{36}$, $b_{23}=a_{23}$, respectively. Thus $\un{A}$, $\un{B}$ are not separated by $P_{3,3}\backslash S_{3,3}$. 

Now we assume that $a_{37}=0$. It follows from $T_{23}$ that $b_{36}=a_{36}$. If $a_{34}=0$, then $\un{A}$, $\un{B}$ are not separated by $P_{3,3}$. If $a_{34}\neq0$, then consequently considering equalities $\si_2(A_3)-\si_2(B_3)=0$, $T_{223}$ we obtain that $b_{32}=a_{32}$, $b_{23}=a_{23}$, respectively.  Hence $\un{A}$, $\un{B}$ are not separated by $P_{3,3}\backslash S_{3,3}$. 

\medskip
\noindent{\bf 2.3.} Assume that the type of $(A_2,B_2)$ is $(W_{\rm III},W_{\rm III})$.  Since $a_{24}\neq0$, consequently considering equalities $T_{113}$, $T_{123}$, $T_{132}$, $\tr(A_3)-\tr(B_3)=0$, $T_{1132}$, $T_{13}$ we obtain that $b_{37}=a_{37}$, $b_{31}=a_{31}$, $b_{35}=a_{35}$, $b_{39}=a_{39}$, $b_{38}=a_{38}$, $b_{34}=a_{34}$, respectively. 

In case $a_{37}\neq0$ we apply consequently equalities $T_{113313}$, $T_{23}$, $T_{331}$, $\si_2(A_3)-\si_2(B_3)=0$ to see that $b_{32}=a_{32}$, $b_{23}=a_{23}$, $b_{36}=a_{36}$, $b_{33}=a_{33}$, respectively. Thus $\un{A}$, $\un{B}$ are not separated by $P_{3,3}$. 

Now we assume that $a_{37}=0$. It follows from $T_{23}$ that $b_{32}=a_{32}$. If $a_{38}=0$, then $\un{A}$, $\un{B}$ are not separated by $P_{3,3}$. If $a_{38}\neq0$, then consequently considering equalities $\si_2(A_3)-\si_2(B_3)=0$, $T_{223}$ we obtain that $b_{36}=a_{36}$, $b_{23}=a_{23}$, respectively.  Hence $\un{A}$, $\un{B}$ are not separated by $P_{3,3}\backslash S_{3,3}$. 

\medskip
\noindent{\bf 2.4.} Assume that the type of $(A_2,B_2)$ is $(W_{\rm IV},W_{\rm IV})$. Since $a_{24}\neq0$ and $a_{28}\neq0$, consequently considering the following equalities $T_{113}$, $T_{13}$, $\tr(A_3)=0$, $\tr(B_3)=0$, $T_{221}$, $\si_2(A_2)-\si_2(B_2)=0$, $T_{1123}$, $\si_2(A_2)=0$, $T_{2213}$, $T_{2231}$, $T_{123}$, $T_{223}$  we can see that $b_{37}=a_{37}$, $b_{38}=a_{34}+a_{38}-b_{34}$, $a_{39}=-a_{31}-a_{35}$, $b_{39}=-b_{31}-b_{35}$, $b_{25}=a_{25}$, $b_{26}=a_{26}$, $b_{34}=a_{34}$, $a_{26}=-a_{25}^2/a_{28}$, $b_{36}=a_{36} + a_{25}(a_{35}-b_{35})/a_{28}$, $b_{32}=a_{32} + a_{25}(a_{31}-b_{31})/a_{28}$, $b_{35}=a_{35} + a_{24}(a_{31}-b_{31})/a_{28}$, $b_{33}=a_{33} + a_{24}a_{25}^2(b_{31}-a_{31})/(a_{24}a_{28}^2)$,  respectively.

In case $b_{31}=a_{31}$ we have $\un{A}=\un{B}$ and the required is proven. Assume that $b_{31}\neq a_{31}$. It follows from $T_{1133}$ and $T_{331}$  that $a_{37}=0$ and $a_{38}=a_{34} - a_{24}a_{34}/a_{28}$, respectively.

If $a_{28}=a_{24}$, then $T_{132}$ implies $a_{24}(a_{31}-b_{31})=0$; a contradiction. Assume that $a_{28}\neq a_{24}$. Then $T_{22123}$ and $T_{2233}$ imply that $a_{25}=0$ and $a_{33}=0$, respectively. It follows from 
$$T_{332} - a_{32} T_{132} = a_{28}(a_{32}-a_{36})(a_{31}-b_{31})=0$$
that $a_{36}=a_{32}$. Consequently applying equalities $T_{3312}$ and $\si_2(A_3)-\si_2(B_3)=0$ we can see that $a_{35}=\frac{1}{2}(a_{31}+b_{31}+a_{24}(b_{31}-a_{31})/a_{28})$ and $a_{31}^2=b_{31}^2$. Hence $b_{31}=-a_{31}\neq0$ and the equality $\det(A_3)=0$ implies $a_{24}a_{31}(a_{24}-a_{28})=0$; a contradiction. 
\end{proof}

\section{Case (b): $A_1= J_2$ and $B_1=J_1$}\label{sectionCaseB}

\begin{lemma}\label{lemmaB}
Consider $\un{A}=(J_2,A_2,A_3)$ and $\un{B}=(J_1,B_2,B_3)$ from $\N_3^3$ that are not separated by $S_{3,3}$.  Then $\un{A}$, $\un{B}$ are not separated by $P_{3,3}$.
\end{lemma}
\begin{proof}
By Lemmas~\ref{lemmaCan2} and \ref{lemmaCan1} , we can assume that $A_2$ and $B_2$ have one of the following types: $W_{\rm I},\ldots,W_{\rm V}$ and $V_{\rm I},\ldots,V_{\rm IV}$, respectively. Equalities $T_{112}$ and $T_{113}$ imply that $a_{27}=a_{37}=0$. 
 
\medskip
\noindent{\bf 1.} Assume $a_{28}\neq 0$. Then equalities $T_{1122}$ and $T_{12}$ imply that $a_{24}=0$ and $b_{24}=a_{28}$, respectively. Therefore, we can see that the type of $(A_2,B_2)$ is $(W_{\rm II},V_{\rm IV})$. Consequently considering equalities $T_{1123}$, $T_{13}$, $T_{221}$, $T_{123}$, $\sigma_2(A_2)-\sigma_2(B_2)=0$, $\tr(A_3)=0$, $\tr(B_3)=0$, $T_{132}$ and $T_{23}$ we obtain that $a_{34}=0$, $b_{34}=a_{38}$, $b_{21}=0$, $b_{31}=a_{35}$, $b_{22}=-b_{23}b_{27}/a_{28}$, $a_{39}=-a_{31}-a_{35}$, $b_{39}=-a_{35}-b_{35}$, $b_{35}=-a_{31}-a_{35}-b_{27}b_{36}/a_{28}$ and $b_{32}=(a_{28}a_{36} - b_{27}b_{33} - b_{28}b_{36} - b_{23}b_{37})/a_{28} + a_{38}b_{23}b_{27}/a_{28}^2$, respectively.
 
In case $b_{28}\neq 0$ we apply consequently equalities $\det(B_2)=0$, $T_{223}$ and $T_{2213}$ to see that $b_{23}=b_{33}=b_{36}=0$. Thus $\un{A}$, $\un{B}$ are not separated by $P_{3,3}$.

Let $b_{28}=0$. Observe that if $b_{23}=0$, then $\un{A}$, $\un{B}$ are not separated by $P_{3,3}\backslash S_{3.3}$. So we can assume that $b_{23}\neq 0$. Considering $T_{2231}$ and $T_{223}$ we obtain that $b_{37}=a_{38}b_{27}/a_{28}$ and $b_{38}=(a_{35}b_{27}  + 2b_{27}b_{35})/a_{28} + b_{27}^2b_{36}/a_{28}^2$, respectively. Hence $\un{A}$, $\un{B}$ are not separated by $P_{3,3}$. 
 
\medskip
\noindent{\bf 2.} Now we assume that $a_{28}=0$. The equality $T_{12}$ implies $b_{24}=a_{24}$. These restrictions show that the only possibilities for the type of the pair $(A_2,B_2)$ are  $(W_{\rm I},V_{\rm I})$,  $(W_{\rm I},V_{\rm II})$,  $(W_{\rm I},V_{\rm III})$ and  $(W_{\rm III},V_{\rm IV})$.  
 
\medskip
\noindent{\bf 2.1.} Assume that the type of $(A_2,B_2)$ is $(W_{\rm I},V_{\rm I})$. Considering equalities $T_{13}$, $\tr(A_3)=0$ and $\tr(B_3)=0$ we obtain that  $b_{34}=a_{34}+a_{38}$, $a_{39}=-a_{31}-a_{35}$ and $b_{39}=-b_{31}-b_{35}$, respectively. In case $a_{34}=0$, the equality $T_{123}$ implies that $b_{26}b_{37}=0$. On the other hand, in case $a_{34}\neq 0$, applying  $T_{1133}$ and $T_{123}$  we obtain that $a_{38} = 0$ and $b_{26}b_{37}=0$, respectively. Hence $\un{A}$, $\un{B}$ are not separated by $P_{3,3}\backslash S_{3,3}$ in both cases.
 
\medskip
\noindent{\bf 2.2.} Assume that the type of $(A_2,B_2)$ is $(W_{\rm I},V_{\rm II})$. Consequently considering equalities $T_{13}$, $\tr(A_3)=0$ and $\tr(B_3)=0$ we obtain that  $b_{34}=a_{34}+a_{38}$, $a_{39}=-a_{31}-a_{35}$ and $b_{39}=-b_{31}-b_{35}$, respectively. 

Let $a_{34}\neq 0$. Then it follows from equalities $T_{1133}$ and $T_{331}$ that $a_{38}=0$ and $b_{35} = a_{31} + a_{35} - b_{31} -b_{36}b_{37}/a_{34}$, respectively.  In case  $b_{25}\neq 0$, the equality $T_{123}$ implies that $a_{34}=b_{25}b_{37}$. On the other hand, in case $b_{25}=0$, the equality $T_{223}$ implies that $b_{23}=0$. Thus $\un{A}$, $\un{B}$ are not separated by $P_{3,3}\backslash S_{3,3}$ in both cases.

Let $a_{34}=0$. In case $b_{25}\neq 0$, it follows from the equality $T_{123}$ that  $a_{38}=b_{25}b_{37}$. On the other hand, in case $b_{25}=0$, the equality $T_{223}$ implies that $a_{38}b_{23}=0$. We can see that $\un{A}$, $\un{B}$ are not separated by $P_{3,3}\backslash S_{3,3}$ in both cases. 
  
\medskip
\noindent{\bf 2.3.} Assume that the type of $(A_2,B_2)$ is $(W_{\rm I},V_{\rm III})$. Consequently considering equalities $T_{221}$, $\det(B_2)=0$, $T_{132}$, $T_{1133}$ and $T_{3321}$ we obtain that $b_{26}$, $b_{21}$, $b_{36}$, $a_{34}a_{38}$, $b_{33}b_{34}$ are equal to zero, respectively. Thus, $\un{A}$, $\un{B}$ are not separated by $P_{3,3}\backslash S_{3,3}$.
  
\medskip
\noindent{\bf 2.4.} Assume that the type of $(A_2,B_2)$ is $(W_{\rm III},V_{\rm IV})$. Since $a_{24}\neq 0$,  considering equalities $T_{221}$, $T_{1132}$, $T_{13}$, $T_{123}$, $\sigma_2(A_2)-\sigma_2(B_2)=0$, $\tr(A_3)=0$, and $\tr(B_3)=0$ we obtain that $b_{21}=0$,  $a_{38}=0$, $b_{34}=a_{34}$, $b_{31}=a_{31}$, $b_{22}=-b_{23}b_{27}/a_{24}$, $a_{39}=-a_{31}-a_{35}$, and $b_{39}=-a_{31}-b_{35}$, respectively. 

In case $b_{28}\neq 0$ it follows from equalities $\det(B_2)=0$, $T_{223}$ and $T_{2213}$ that $b_{23}=b_{33}=b_{36}=0$; therefore, $\un{A}$, $\un{B}$ are not separated by $P_{3,3}\backslash S_{3,3}$. Thus we can assume that $b_{28}=0$. In case $b_{23}=0$ we obviously have that $\un{A}$, $\un{B}$ are not separated by $P_{3,3}\backslash S_{3.3}$. On the other hand, in case $b_{23}\neq 0$ equalities $T_{2231}$, $T_{132}$, $T_{223}$ imply that $b_{37}=a_{34}b_{27}/a_{24}$, $b_{35} = a_{35} - b_{27}b_{36}/a_{24}$ and $b_{38}=(a_{31}b_{27} + 2a_{35}b_{27})/a_{24} - b_{27}^2b_{36}/a_{24}^2$, respectively. Hence $\un{A}$, $\un{B}$ are not separated by $P_{3,3}\backslash S_{3,3}$. 
\end{proof}

\section{Case (c): $A_1= J_1$ and $B_1=J_1$.}\label{sectionCaseC}

\begin{lemma}\label{lemmaC}
Consider $\un{A}=(J_1,A_2,A_3)$ and $\un{B}=(J_1,B_2,B_3)$ from $\N_3^3$ that are not separated by $S_{3,3}$.  Then $\un{A}$, $\un{B}$ are not separated by $P_{3,3}$.
\end{lemma}
\begin{proof}
By Lemma~\ref{lemmaCan1}, we can assume that $A_2$ and $B_2$ have one of the following types: $V_{\rm I},\ldots,V_{\rm IV}$. The equality $T_{12}$ implies that $b_{24}=a_{24}$, therefore we can see that the only possibilities for the type of $(A_2, B_2)$ are $(V_{\rm I},V_{\rm I})$, $(V_{\rm I},V_{\rm II})$, $(V_{\rm I},V_{\rm III})$, $(V_{\rm II},V_{\rm II})$, $(V_{\rm II},V_{\rm III})$, $(V_{\rm III},V_{\rm III})$ and $(V_{\rm IV},V_{\rm IV})$. Furthermore, consequently considering the equalities $T_{13}$, $\tr(A_3)=0$, and $\tr(B_3)=0$ we obtain that $b_{34}=a_{34}$, $a_{39}=-a_{31}-a_{35}$, and $b_{39}=-b_{31}-b_{35}$.

\medskip
\noindent{\bf 1.} Assume that the type of $(A_2,B_2)$ is $(V_{\rm I},V_{\rm I})$. Note that if $b_{26}=0$, then equality $T_{123}$ implies that $a_{26}=0$ or $a_{37}=0$ and in both cases we have that  $\un{A}$, $\un{B}$ are not separated by $P_{3,3}\backslash S_{3,3}$. On the other hand, in case $b_{26}\neq 0$ considering the equalities $T_{23}$ and $T_{123}$ we obtain $b_{38}=(a_{22}a_{34}+a_{23}a_{37}+a_{26}a_{38}-a_{34}b_{22}-b_{23}b_{37})/b_{26}$ and $b_{37}=a_{26}a_{37}/b_{26}$, respectively. If  $a_{26}=0$ or $a_{37}=0$, then $\un{A}$, $\un{B}$ are not separated by $P_{3,3}\backslash S_{3,3}$. Thus we can assume that  $a_{26},a_{37}\neq 0$. Considering the equalities $T_{223}$ and $T_{331}$, we obtain $b_{22}=a_{22}$ and $a_{36}=a_{34}(-a_{31}-a_{35}+b_{31}+b_{35}) /a_{37} + a_{26} a_{36} /b_{26}$. Therefore,  $\un{A}$, $\un{B}$ are not separated by $P_{3,3}\backslash S_{3,3}$.

\medskip
\noindent{\bf 2.} Assume that the type of $(A_2,B_2)$ is $(V_{\rm I},V_{\rm II})$. At first, assume that $a_{37}\neq 0$. Then considering the equalities $T_{331}$ and $T_{123}$ we obtain that $a_{36}=(a_{34}(-a_{31}-a_{35}+b_{31}+b_{35}) + b_{36}b_{37})/a_{37}$, $a_{26}=b_{25}(a_{34}-b_{25}b_{37})/a_{37}$, respectively. Thus the equality $T_{223}$ implies that $b_{23}-a_{22}b_{25}=0$ or $a_{34}-b_{25}b_{37}=0$, and therefore  $\un{A}$, $\un{B}$ are not separated by $P_{3,3}\backslash S_{3,3}$.

On the other hand, assume that $a_{37}=0$. If $b_{23}=0$, then $\un{A}$, $\un{B}$ are not separated by $P_{3,3}\backslash S_{3,3}$. Furthermore, if $b_{23}\neq 0$, then the equality $T_{223}$ implies $a_{34}=b_{25}b_{37}$ and we also have that $\un{A}$, $\un{B}$ are not separated by $P_{3,3}\backslash S_{3,3}$.

\medskip
\noindent{\bf 3.} Assume that the type of $(A_2,B_2)$ is $(V_{\rm I},V_{\rm III})$. Considering the equalities $T_{221}$, $\det(B_2)=0$, and $T_{132}$, we obtain $b_{26}=b_{21}=b_{36}=0$. It follows from  $T_{123}$ and $T_{3321}$ that $a_{26}a_{37}=0$ and $a_{34}b_{33}=0$, respectively. We can conclude that $\un{A}$, $\un{B}$ are not separated by $P_{3,3}\backslash S_{3,3}$.

\medskip
\noindent{\bf 4.} Assume that the type of $(A_2,B_2)$ is $(V_{\rm II},V_{\rm II})$. It follows from the equality $T_{23}$ that \linebreak
$b_{36} = a_{36} + a_{25}(a_{31}+2a_{35}) -a_{25}^2a_{38} +a_{23}a_{37} - b_{25}(b_{31}+2b_{35}) \ -b_{23}b_{37}+b_{25}^2b_{38}$. 

Assume that $a_{25}\neq 0$. Then the equality $T_{123}$ implies that $a_{37}=(a_{25}a_{34}+b_{25}(-a_{34}+b_{25}b_{37}))/a_{25}^2$, and $T_{223}$ implies that $a_{25}b_{23}-a_{23}b_{25}=0$ or $-a_{34}+b_{25}b_{37}=0$. Furthermore, if $-a_{34}+b_{25}b_{37}=0$, then  $\un{A}$, $\un{B}$ are not separated by $P_{3,3}\backslash S_{3,3}$. On the other hand, if $-a_{34}+b_{25}b_{37}\neq0$, then $b_{23}=a_{23}b_{25}/a_{25}$ and we can express $a_{38}$ from the equality $T_{3312}$. Since  $T_{332231}=(a_{23}b_{25}(a_{34}-b_{25}b_{37})/a_{25})T_{331}$, we can conclude that $\un{A}$, $\un{B}$ are not separated by $P_{3,3}\backslash S_{3,3}$.  

Now we assume that $a_{25}=0$. If $a_{34}=b_{25}=0$, then  $\un{A}$, $\un{B}$ are not separated by $P_{3,3}\backslash S_{3,3}$. If $a_{34}\neq0$ and $b_{25}=0$, then the equalities $T_{223}$ and $T_{331}$ imply that $b_{23}=a_{23}$ and $a_{31}=(a_{34}(-a_{35}+b_{31}+b_{35})+(a_{37}-b_{37})(-a_{36}+a_{23}b_{37}))/a_{34}$; therefore, we also have that $\un{A}$, $\un{B}$ are not separated by $P_{3,3}\backslash S_{3,3}$. Finally, if $b_{25}\neq 0$, then $T_{123}$ implies $a_{34}=b_{25} b_{37}$ and  $T_{223}$ implies that $a_{23}=0$ or $b_{37}=0$; thus, $\un{A}$, $\un{B}$ are not separated by $P_{3,3}\backslash S_{3,3}$.

\medskip
\noindent{\bf 5.}  Assume that the type of $(A_2,B_2)$ is $(V_{\rm II},V_{\rm III})$. Considering the equalities $T_{221}$, $\det(B_2)=0$, and $T_{132}$ we obtain $b_{26}=b_{21}=b_{36}=0$. Hence,  $T_{332231} = (a_{31}a_{34}+a_{34}a_{35}+a_{36}a_{37}) T_{223} + a_{34}b_{22} T_{3321}$. Thus, the equalities $T_{223}$ and $T_{3321}$ imply that $\un{A}$, $\un{B}$ are not separated by $P_{3,3}\backslash S_{3,3}$.

\medskip
\noindent{\bf 6.} Assume that the type of $(A_2,B_2)$ is $(V_{\rm III},V_{\rm III})$. It follows from equalities $T_{221}$ and $T_{132}$ that $b_{26}=a_{26}$ and $b_{36}=a_{21}a_{34}+a_{36}-a_{34}b_{21}$, respectively. 

Assume $a_{26}=0$. Then consequently considering equalities $\det(A_2)=0$, $\det(B_2)=0$, and $T_{23}$ we obtain that $a_{21}=b_{21}=0$ and $b_{33}=a_{33}+a_{22}a_{34}-a_{34}b_{22}$. Hence, if $a_{34}=0$, then $\un{A}$, $\un{B}$ are not separated by $P_{3,3}\backslash S_{3,3}$. On the other hand, if $a_{34}\neq 0$, then the equalities $T_{331}$ and $T_{3321}$ imply that $b_{35}=(a_{31}a_{34}+a_{34}(a_{35}-b_{31})+a_{36}(a_{37}-b_{37}))/a_{34}$ and $a_{22}=(-a_{31}a_{36}+a_{34}^2b_{22}+a_{36}b_{31})/a_{34}^2$, respectively. Finally, $T_{223}$ implies that $a_{36}=0$ or $a_{31}-b_{31}=0$. In both cases we have  that $\un{A}$, $\un{B}$ are not separated by $P_{3,3}\backslash S_{3,3}$.

Now we assume that $a_{26}\neq 0$. Consequently applying equalities $\det(A_2)=0$, $\det(B_2)=0$, $T_{123}$, $T_{2213}$, $T_{2231}$, and $T_{23}$ we obtain that 
$a_{22}=a_{21}^3/a_{26}$, $b_{22}=b_{21}^3/a_{26}$,\linebreak $b_{37}=(-a_{21}a_{34}+a_{34}b_{21}+a_{26}a_{37})/a_{26}$, $b_{35}=((a_{21}-b_{21})(a_{36}-a_{34}b_{21})+a_{26}a_{35})/a_{26}$, $b_{31}=(a_{21}a_{34}(a_{21}-b_{21})+a_{26}(a_{31}-a_{21}a_{37}+a_{37}b_{21}))/a_{26}$ and $b_{38}=(a_{21}^3a_{34}+a_{21}b_{21}(a_{36}-a_{34}b_{21})+b_{21}^2(-a_{36}+a_{34}b_{21})-a_{21}^2(a_{26}a_{37}+a_{34}b_{21})+a_{21}a_{26}(a_{31}-a_{35}+a_{37}b_{21})+a_{26}(a_{33}-a_{31}b_{21}+a_{35}b_{21}-b_{33})+a_{26}^2a_{38})/a_{26}^2$, respectively. If $a_{34}\neq 0$, then considering $T_{3312}$ and $T_{223}$ we can express $b_{33}$ and $b_{32}$, respectively; therefore, $\un{A}$, $\un{B}$ are not separated by $P_{3,3}$. Thus, we can assume that $a_{34}=0$. Considering equalities $T_{22123}$ and $T_{223}$ we can express $b_{33}$ and $b_{32}$, respectively;  therefore, $\un{A}$, $\un{B}$ are not separated by $P_{3,3}$.

\medskip
\noindent{\bf 7.}  Assume that the type of $(A_2,B_2)$ is $(V_{\rm IV},V_{\rm IV})$. Since $a_{24}\neq 0$, considering the equalities $T_{221}$, $T_{123}$, $\sigma_2(A_2)=0$, $\sigma_2(B_2)=0$, $T_{132}$, and $T_{23}$ we obtain that  $b_{21}=a_{21}$, $b_{31}=a_{31}$,\linebreak $a_{22}=(-a_{21}^2-a_{23}a_{27})/a_{24}$, $b_{22}=(-a_{21}^2-b_{23}b_{27})/a_{24}$, $b_{35}=(a_{27}a_{36}+a_{24}a_{35}-b_{27}b_{36})/a_{24}$ and 
$b_{32}=(-a_{23}a_{27}a_{34}-a_{21}a_{27}a_{36}+a_{34}b_{23}b_{27}+a_{24}^2a_{32}+a_{21}b_{27}b_{36}+a_{24}(a_{27}a_{33}+a_{28}a_{36}+a_{23}a_{37}-b_{27}b_{33}-b_{28}b_{36}-b_{23}b_{37}))/a_{24}^2$, respectively.

\medskip
\noindent{\bf 7.1.} Assume $a_{23}\neq 0$. Considering equalities $\det(A_2)=0$ and $T_{2231}$ we have that \linebreak $a_{28}=(a_{21}^3+a_{21}a_{23}a_{27})/(a_{23}a_{24})$ and $a_{37}=(a_{23}a_{27}a_{34}-a_{34}b_{23}b_{27}+a_{24}b_{23}b_{37})/(a_{23}a_{24})$. 

Let $b_{23}=0$. Then the equality $\det(B_2)=0$ implies that $a_{21}=0$ and we can express $a_{38}$ from the equality $T_{223}$. If $b_{28}=0$, then $\un{A}$, $\un{B}$ are not separated by $P_{3,3}\backslash S_{3,3}$. On the other hand, if $b_{28}\neq 0$, then the equalities $T_{2213}$ and $T_{22123}$ imply that $b_{36}=b_{33}=0$, and therefore $\un{A}$, $\un{B}$ are not separated by $P_{3,3}$.

Let $b_{23}\neq 0$. Then it follows from the equality $\det(B_2)=0$ that $b_{28}=(a_{21}^3+a_{21}b_{23}b_{27})/(a_{24}b_{23})$. In case  $a_{21}=0$ we can express $b_{38}$ from the equality
$T_{223}$ and we can see that  $T_{332231}=(-a_{24}a_{27}(a_{31}+2a_{35})-a_{27}^2a_{36}+a_{24}^2a_{38})(a_{23}/a_{24})T_{331}$; thus, $\un{A}$, $\un{B}$ are not separated by $P_{3,3}\backslash S_{3,3}$. On the other hand, if $a_{21}\neq 0$, then equalities $T_{2213}$, $T_{22123}$ imply that $b_{36}=a_{36}b_{23}/a_{23}$, $b_{33}=a_{33}b_{23}/a_{23}$ and we can express  $b_{38}$ from $T_{223}$. Therefore,  $\un{A}$, $\un{B}$ are not separated by $P_{3,3}$.

\medskip
\noindent{\bf 7.2.} Assume that $a_{23}=0$. Then the equality $\det(A_2)=0$ implies that  $a_{21}=0$. 

Let $b_{23}\neq 0$. Since $a_{24}\neq 0$, considering equalities $\det(B_2)=0$, $T_{2231}$ we obtain $b_{28}=0$, $b_{37}=a_{34}b_{27}/a_{24}$ and then we can express $b_{38}$ from $T_{223}$. If $a_{28}=0$, then $\un{A}$, $\un{B}$ are not separated by $P_{3,3}\backslash S_{3,3}$. Thus we can assume that  $a_{28}\neq 0$. It follows from $T_{2213}$ and $T_{22123}$ that $a_{36}=a_{33}=0$,  and therefore $\un{A}$, $\un{B}$ are not separated by $P_{3,3}$.

On the other hand, let $b_{23}=0$. If $b_{28}\neq 0$, then considering the equalities $T_{223}$ and $T_{2213}$ we obtain $b_{33}=a_{28}a_{33}/b_{28}$ and $b_{36}=a_{28}a_{36}/b_{28}$, respectively; thus, $\un{A}$, $\un{B}$ are not separated by $P_{3,3}\backslash S_{3,3}$. Thus we can assume that $b_{28}=0$. It is easy to
see that if $a_{28}=0$, then $\un{A}$, $\un{B}$ are not separated by $P_{3,3}\backslash S_{3,3}$. If $a_{28}\neq 0$, then the equalities $T_{223}$ and $T_{2213}$ imply that $a_{33}=a_{36}=0$, and therefore $\un{A}$, $\un{B}$ are not separated by $P_{3,3}\backslash S_{3,3}$.
\end{proof}

\section{Cases (d) and (e): $B_1=0$. }\label{sectionCaseDE}

\begin{lemma}\label{lemmaD}
Consider $\un{A}=(J_1,A_2,A_3)$ and $\un{B}=(0,B_2,B_3)$ from $\N_3^3$ that are not separated by $S_{3,3}$.  Then $\un{A}$, $\un{B}$ are not separated by $P_{3,3}$.
\end{lemma}
\begin{proof} 
By Lemma~\ref{lemmaCan1}, we can assume that $A_2$ has one of the following types: $V_{\rm I},\ldots,V_{\rm IV}$.  Equalities  $T_{12}$ and $T_{13}$ imply that $a_{24}=a_{34}=0$. Then it follows from equality $T_{221}$ that $a_{26}=0$ or $a_{27}=0$.  

Assume $a_{26}= 0$ and $a_{27}\neq 0$. Thus the type of $A_2$ is $V_{\rm III}$. The equality  $T_{132}$ implies that $a_{36}=0$.  Hence $\un{A}$, $\un{B}$ are not separated by $P_{3,3}\backslash S_{3,3}$. 

Assume $a_{27}=0$. Thus the type of $A_2$ is either $V_{\rm I}$ or  $V_{\rm II}$. In both cases the equality $T_{331}$ implies that $a_{36}a_{37}=0$.  Hence $\un{A}$, $\un{B}$ are not separated by $P_{3,3}\backslash S_{3,3}$.  
\end{proof}

\begin{lemma}\label{lemmaE}
Consider $\un{A}=(J_2,A_2,A_3)$ and $\un{B}=(0,B_2,B_3)$ from $\N_3^3$ that are not separated by $S_{3,3}$.  Then $\un{A}$, $\un{B}$ are not separated by $P_{3,3}$.
\end{lemma}
\begin{proof}
By Lemma~\ref{lemmaCan2}, we can assume that $A_2$ has one of the following types: $W_{\rm I},\ldots,W_{\rm V}$. Consequently considering equalities $T_{112}$, $T_{12}$, $T_{1122}$ we can see that $a_{27}=0$, $a_{28}=-a_{24}$, $a_{24}=0$, respectively. Then the type of $A_2$ is  $W_{\rm I}$. Consequently applying equalities $T_{13}$, $T_{113}$, $T_{1133}$ we obtain that $a_{38}=-a_{34}$, $a_{37}=0$, $a_{34}=0$, respectively. Hence $\un{A}$, $\un{B}$ are not separated by $P_{3,3}\backslash S_{3,3}$. 
\end{proof}

\end{document}